\documentclass{article}
\usepackage[letterpaper, portrait, margin=2.5cm]{geometry}

\usepackage{verbatim}
\usepackage[utf8]{inputenc}
\usepackage{amsmath,amssymb,amsthm}
\usepackage{tikz-cd}
\usepackage{bbm}
\usepackage{comment}
\usepackage{stmaryrd}
\usepackage[new]{old-arrows}
\usepackage{graphicx}
\graphicspath{ {./Images/} }
\usepackage{todonotes}

\usepackage [russian,english]{babel}

\usepackage[maxbibnames=99]{biblatex}
\newcommand{\lbm}{\left[ \begin{matrix}}
\newcommand{\rem}{\end{matrix} \right]}

\theoremstyle{definition}
\newtheorem{definition}{Definition}[section]

\newtheorem{lemma}[definition]{Lemma}
\newtheorem{theorem}[definition]{Theorem}

\title{On the Problem of Resolvents}
\author{A Translation by Alexander J. Sutherland}
\date{\today}

\begin{document}

\maketitle

\vspace{16pt}

\section{Original Bibliographic Information}
I thank Ignat Soroko for helpful comments on the bibliographic information.

\subsection{Russian}
\begin{itemize}
	\item \textbf{Author:} \foreignlanguage{russian}{Г. Н. Чеботарев}
	\item \textbf{Title:} \foreignlanguage{russian}{К Пpoблeмe Рeзoльвeнт}
	\item \textbf{Year:} 1954
	\item \textbf{Language:} \foreignlanguage{russian}{русский}
	\item \textbf{Publisher:} \foreignlanguage{russian}{Учeныe Записки Казанского Государствeнного Унивeрситeта им. В. И. Ульянова-Лeнина}
\end{itemize}

\subsection{English}
\begin{itemize}
	\item \textbf{Author:} G.N. Chebotarev
	\item \textbf{Title:} On the Problem of Resolvents
	\item \textbf{Year:} 1954
	\item \textbf{Language:} Russian
	\item \textbf{Publisher:} Scientific Proceedings of the V.I. Ulyanov-Lenin Kazan State University
\end{itemize}

\vfill

\begin{center}
This work was supported by the National Science Foundation under Grant No. DMS-1944862.
\end{center}

\newpage
\section{The Translation} 

\subsection*{On the Problem of Resolvents}
\footnote{
Translator's Note: This is a translation of the original mathematics. In particular, errors in the text have not been fixed. The errors in question come from considering intersections in affine spaces instead of in projective spaces.} 
Consider an equation of $n^{th}$ degree
\begin{equation}
	f(x) = x^n + a_1x^{n-1} + \cdots + a_n = 0
\end{equation}

\noindent
whose coefficients $a_1,\dotsc,a_n$ are indeterminates. Substituting
\begin{align*}
	y = t_0 + t_1x + \cdots + t_{n-2}x^{n-2} + x^{n-1}
\end{align*}

\noindent
into $f(x)=0$ yields an $n^{th}$ degree equation
\begin{equation}
	y^n + C_1y^{n-1} + C_2y^{n-2} + \cdots + C_n = 0
\end{equation}

\noindent
where the coefficients $C_1,\dotsc,C_n$ depend rationally on $a_1,\dotsc,a_n$ and are polynomials in $t_0,\dotsc,t_{n-2}$ of respective degrees $1,\dotsc,n$. \\

Equation (2) is called an \textit{$s$-parameter resolvent} of equation (1) if its coefficients $C_1,\dotsc,C_n$ are rational functions of $s$ parameters $v_1,\dotsc,v_s$ and the coefficients $t_0,\dotsc,t_{n-2}$ of the Tschirnhaus transformation depend rationally on $a_1,\dotsc,a_n$ and the roots of some auxiliary equations (secondary resolvents), which themselves admit $s$-parameter resolvents. \\

It is easy to show that if one does not limit the degree of the secondary resolvents, the $s$-parameter resolvent of equation (1) can be put in the particular form
\begin{equation}
	y^n + C_{n-s}y^s + C_{n-s+1}y^{s-1} + \cdots + C_{n-1}y + 1 = 0.
\end{equation}

In fact, let
\begin{align*}
	g(z) = z^n + B_1z^{n-1} + \cdots + B_n = 0
\end{align*}

\noindent
be an $s$-parameter resolvent of equation (1), i.e. $B_1,\dotsc,B_n$ are rational functions $s$ of parameters $v_1,\dotsc,v_s$. \\

Take the new Tschirnhaus transformation
\begin{align*}
	y = \tau_0 + \tau_1z + \cdots + \tau_{n-2}z^{n-2} + z^{n-1}.
\end{align*}

The coefficients $C_1,\dotsc,C_n$ of the equation that $y$ satisfies are polynomials of the corresponding degree in the $\tau_0,\dotsc,\tau_{n-2}$ and, moreover, depend rationally on the coefficients $B_1,\dotsc,B_n$ and therefore on the parameters $v_1,\dotsc,v_s$. Setting $C_1,\dotsc,C_{n-s-1}$ equal to zero and $C_n$ equal to 1 and composing the results of these equations (in which $\tau_0,\dotsc,\tau_{n-2}$ are unknown), we obtain a chain of auxiliary equations whose coefficients depend on $s$ parameters. \\

D. Hilbert (1) showed in his article "On Equations of the Ninth Degree" that an equation of the ninth degree admits a resolvent that depends on 4 parameters. His method of obtaining this resolvent is as follows. \\

The coefficients $t_0,\dotsc,t_{n-2}$ of the Tschirnhaus transformation are considered as the coordinates of a point in the space $T$ taking values from the field of rational functions in $a_1,\dotsc,a_n$ and its algebraic extensions. The equations $C_1=0, C_2=0, C_3=0$ and $C_4=0$ determine hypersurfaces in the space $T$ of degrees 1, 2, 3, and 4, respectively. Finding a 4-parameter resolvent of an equation of $9^{th}$ degree reduces to finding a common point of these hypersurfaces by solving a chain of algebraic equations that admit $\leq 4$-parameter resolvents. Substituting $y = \sqrt[n]{C_n} z$ makes the final term a unit. \\

This problem is solved by Hilbert as follows. A three-dimensional hyperplane is found that entirely belongs to the hypersurfaces $C_1=0, C_2=0$. On it, the surface $C_3=0$ cuts out a cubic surface, which, as you know, always contains lines that lie entirely on it; to find these lines, one has to solve an equation of the 27th degree which depends only on four parameters, since the equation of the cubic surface allows a special technique based on the subtle properties of cubic quarternary forms (by summing up five cubes that are the roots of one equation of fifth-degree) which leads to a form that depends on 4 parameters. The intersection of the line just found with the hypersurface $C_4=0$ determines the desired point. Thus, to construct a 4-parameter resolvent of an equation of the $9^{th}$ degree, in addition to a series of equations of degree between 2 and 5, it is necessary to solve an equation of the $27^{th}$ degree, which is greater than the degree of the original equation. \\

In the work "On the Application of Tschirnhaus Transformations to the Reduction of Algebraic Equations" (2), A. Wiman show that to obtain an $(n-5)$-parameter resolvent of an equation of degree $n \geq 10$, it is sufficient to solve several auxiliary one-parameter equations of degree no higher than 4.  To do this, he moves the origin to an intersection point of the hypersurfaces $C_1=0, C_2=0, C_3=0$ and reduces the last two equations to the form
\begin{align*}
	0 &= C_2 = \phi_2\\
	0 &= C_3 = \psi_2 + \psi_3
\end{align*}

\noindent
where $\phi_2$ and $\psi_2$ are quadratic homogeneous forms and $\psi_3$ is a cubic homogeneous form, and using elegant geometric considerations, searches for a two-dimensional plane that lies entirely in both the hypercones $\phi_2=0$ and $\psi_2=0$. The intersection of this plane with the cubic cone $\psi_3=0$ gives a straight line belonging to the surfaces $C_1=0$, $C_2=0$, $C_3=0$. \\

For the case $n=9$, Wiman proves that in order to obtain a 4-parameter resolvent, it is sufficient to only solve one auxiliary equation of 5th degree (which has a one-parameter resolvent) in addition to the the equations of degrees 1-4. To do this, he performs a linear transformation (which is determined by solving an equation of fifth degree) which simultaneously diagonalizes the forms $\phi_2=0$ and $\psi_2=0$ and defines a one-parameter family of two-dimensional cones, all points of which belong to both cones $\phi_2=0$ and $\psi_2=0$.  By solving an equation of fourth degree, we find the value of the parameter at which the cone of the family splits into a pair of planes. \\

Applying the method of Wiman, we can find an $(n-6)$-parameter resolvent of an equation of degree $n \geq 77$. In this article, an attempt is made to slightly modify this method, as a result of which, the $(n-6)$-parameter resolvent can be constructed for equations of degree $n \geq 21$. \\

Following Wiman, we consider the space $T_{n-1}$ of the parameters $t_0,\dotsc,t_{n-2}$ and the hypersurfaces $C_1=0, C_2=0, C_3=0, C_4=0, C_5=0$ in this space. We move the origin to a point common to the hypersurfaces $C_1=0, C_2=0, C_3=0$, which can be determined by  solving auxiliary equations of the second and third degree. Now, in the equations of the first three surfaces, the free terms disappear and these equations can be written as follows:
\begin{align*}
	0 &= C_1\\
	0 &= C_2 = \phi_1 + \phi_2\\
	0 &= C_3 = \psi_1 + \psi_2 + \psi_3
\end{align*}

\noindent
where $C_1, \phi_1, \psi_1$ are linear forms of the parameters $t_i$, $\psi_2, \phi_2$ are quadratic [forms], and $\psi_3$ is a cubic. The equations $\phi_1=0$ and $\psi_1=0$ determine the hyperplanes tangent to the hypersurfaces $C_2=0$ and $C_3=0$ at the origin. The intersection of these hyperplanes with the hyperplane $C_1=0$ determines the space $T_{n-4}$, in which the equations of the hypersurfaces cut out by $C_2=0$ and $C_3=0$ will have the form
\begin{align*}
	0 &= C_2' = \phi_2'\\
	0 &= C_3' = \psi_2' + \psi_3'.
\end{align*}

We show that for $n \geq 19$, there exists a two-dimensional plane belonging entirely to the hypersurfaces $C_2'=0$ and $C_3'=0$ in the space $T_{n-4}$. 

\begin{lemma}
	Two $(3k-1)$-dimensional quadratic cones with a common vertex in $3k$-dimensional space share a whole $k$-dimensional plane passing through the vertex of the cones. 
\end{lemma}

\begin{proof}
	We proceed by induction. Find straight-line generators common to both cones (for this, it is enough to intersect both cones with any two-dimensional plane that does not pass through their vertex, find the intersection point of the two quadrics cut by the cones on the plane - which requires solving an equation of fourth degree - and connect the found point to the vertex of the cones). Take a plane of dimension $3k-1$ that does not pass through vertex of the cones. On this hyperplane, our cones will cut out two hypersurfaces of degree 2, the common point of which is rationally defined as the intersection of the hyperplane and the previously found common [straight-line] generator of the cones. \\
	
	The intersection of the [original] hyperplane and the two [tangent] hyperplanes touching these hypersurfaces at their common point defines a space of $3k-3$ dimensions, in which out cones cut out a pair of cones of $3k-4$ dimensions with a common vertex, which contains a generic $k-1$ dimensional linear space by the inductive hypothesis. The desired $k$-dimensional subspace common to both cones is defined as the space passing through the $(k-1)$-dimensional space and the vertex of the cones.
\end{proof}

\begin{lemma}
	A cubic four-dimensional cone in five-dimensional space contains a two-dimensional plane passing through the top of the cone which lies entirely in the cone. 
\end{lemma}

\begin{proof}
	Intersect our cone with a four-dimensional plane that does not pass through its top. The cone will cut out a three-dimensional cubic hypersurface on it. We find a point on this surface (for which it is enough to solve one equation of the third degree) and construct a three-dimensional hyperplane that is tangent to the surface at this point. If, after moving the point to the origin, the equation of the hypersurface has the form
	\begin{align*}
		\phi_1 + \phi_2 + \phi_3 = 0,
	\end{align*}
	
	\noindent
	then the equation of the tangent hyperplane will be
	\begin{align*}
		\phi_1=0.
	\end{align*}
	
	\noindent
	Consider the intersection of our surface with the given tangent hyperplane. Obviously, the equation of this intersection will have the form
	\begin{align*}
		\phi_2' + \phi_3' = 0
	\end{align*}
	
	\noindent
	where $\phi_2'$ and $\phi_3'$ are forms of second and third degree, respectively, in three variables. We consider the quadratic and cubic cones
	\begin{align*}
		\phi_2'=0 \text{ and } \phi_3'=0
	\end{align*}
	
	\noindent
	with a common vertex in three-dimensional space. These cones have a common straight line generator and it suffices to solve an equation of the sixth degree to find it (determine the intersection point of the quadric and cubic cut out by the cones on any two-dimensional plane not passing through their vertex and connect it to the vertex of the cones). The two-dimensional plane which passes through the [original] vertex and through the straight line just found lies entirely in the original cubic cone.
\end{proof}

\begin{theorem}
	The general algebraic equation of degree $n \geq 21$ admits an $(n-6)$-parameter resolvent.
\end{theorem}

\begin{proof}
	As above, consider the hypersurfaces
	\begin{align*}
		C_1=0, \ C_2=0, \ C_3=0, \ C_4=0, \ C_5=0
	\end{align*}
	
	\noindent
	in the space $T_{n-1}$. We move the origin to a common point of the hypersurfaces $C_1=0, \ C_2=0, \ C_3=0$. We construct tangent hyperplanes to the hypersurfaces $C_2=0$ and $C_3=0$ at the origin and consider the space $T_{n-4}$, which is the intersection of these hyperplanes with the hyperplane $C_1=0$. In $T_{n-4}$, our surfaces $C_2'=0$ and $C_3'=0$ are defined by equations of the form
	\begin{align*}
		0 &= C_2' = \phi_2,\\
		0 &= C_3' = \psi_2 + \psi_3.
	\end{align*}
	
	\noindent
	By virtue of Lemma 1 and as $n-4 \geq 15$, the two cones
	\begin{align*}
		\phi_2=0, \psi_2=0
	\end{align*}
	
	\noindent
	have a common 5-dimensional linear subspace. According to Lemma 2, the cubic cone $\psi_3'=0$ in this subspace contains a two-dimensional plane and according to the above, we do not need to solve any equations above the sixth degree to find one. \\
	
	 The hypersurfaces
	 \begin{align*}
	 	C_4=0 \text{ and } C_5=0
	 \end{align*}
	 
	 \noindent 
	 cut out curves of the $4^{th}$ and $5^{th}$ degree are cut out on this plane, the intersection point of which can be found by solving an equation of the $20^{th}$ degree, which, according to Wiman, has a resolvent that depends on no more than 15 parameters. However, $n-6 \geq 15$, which proves the theorem. 
\end{proof}

This technique allows us to state the existence of $(n-7)$-parameter resolvents of a general equation of degree $n \geq 121$.

\subsection*{Literature}
(1) D. Hilbert. \"Uber die Gleichung neunten Grades. Ges. Abh., Bd. II, S. 393.\\
\hspace{-14pt} (2) A. Wiman. \"Uber die Anwendung der Tschirnhausentransformationen auf die Reduktion algebraischer Gleichungen. Nova Acta Regiae Societatis Scientiarum Upsaliensis, volumen extra ordinem 1927. \\
\\
Department of Algebra \hfill Received January 19, 1953.


\end{document}